\documentclass[leqno]{siamltex}
\usepackage{amsmath}
\usepackage{amssymb}
\usepackage{graphicx}

\newcommand{\bq}{{\bf q}}
\newcommand{\bn}{{\bf n}}
\newcommand{\bx}{{\bf x}}

\newcommand{\pT}{{\partial T}}

\def\bbQ{\mathbb{Q}}
\def\T{{\mathcal T}}
\def\E{{\mathcal E}}

\def\l{{\langle}}
\def\r{{\rangle}}
\def\3bar{{|\hspace{-.02in}|\hspace{-.02in}|}}

\newtheorem{algorithm}{Weak Galerkin Algorithm}

\title{A conforming DG method for the biharmonic equation on polytopal meshes}

\author{Xiu Ye\thanks{Department of
Mathematics, University of Arkansas at Little Rock, Little Rock, AR
72204 (xxye@ualr.edu). This research was supported in part by
National Science Foundation Grant DMS-1620016.}
\and
Shangyou Zhang\thanks{Department of
Mathematical Sciences, University of Delaware, Newark, DE 19716 (szhang@udel.edu).}
}

\begin{document}

\maketitle

\begin{abstract}
A conforming discontinuous Galerkin finite element method is introduced for solving the biharmonic equation.
This method, by its name, uses discontinuous approximations and keeps simple formulation of the conforming finite element method at the same time. The ultra simple formulation of the method will reduce programming complexity in practice.
Optimal order error estimates in a discrete $H^2$ norm is established for the
corresponding finite element solutions. Error estimates in the $L^2$ norm are also derived with a sub-optimal order of
convergence for the lowest order element and an optimal order of
convergence for all high order of elements. Numerical results are
presented to confirm the theory of convergence.
\end{abstract}


\pagestyle{myheadings}

\section{Introduction}

We consider the biharmonic equation of the form
\begin{eqnarray}
\Delta^2 u&=&f\quad \mbox{in}\;\Omega,\label{pde}\\
u&=&0\quad\mbox{on}\;\partial\Omega,\label{bc-d}\\
\frac{\partial u}{\partial
n}&=&0\quad\mbox{on}\;\partial\Omega,\label{bc-n}
\end{eqnarray}
where $\Omega$ is a bounded polytopal domain in $\mathbb{R}^d$.

The weak formulation of the boundary value problem (\ref{bc-d}) and (\ref{bc-n}) is  seeking $u\in H^2_0(\Omega)$ satisfying
\begin{equation}\label{wf}
(\Delta u, \Delta v) = (f, v)\qquad \forall v\in H_0^2(\Omega).
\end{equation}

The $H^2$ conforming finite element method for the problem (\ref{pde})-(\ref{bc-n}) keeps the same simple form as in (\ref{wf}): find $u_h\in V_h\subset H^2_0(\Omega)$
such that
\begin{eqnarray}
(\Delta u_h,\Delta v)=(f,v)\quad \forall v\in V_h.\label{cfe}
\end{eqnarray}
However, it is known that $H^2$-conforming
methods  require $C^1$-continuous piecewise polynomials
on a simplicial meshes, which imposes
difficulty in practical computation. Due to the complexity in the
construction of $C^1$-continuous elements, $H^2$-conforming finite
element methods are rarely used in practice for solving the
biharmonic equation.

An approach of avoiding construction of $H^2$-conforming elements is to use discontinuous approximations. Due to the flexibility of discontinuous Galerkin (DG) finite element methods in element constructions and in mesh generations, many finite element methods have been developed using totally discontinuous polynomials. Here we are only interested in interior penalty discontinuous Galerkin (IPDG)  methods since the proposed the method shares the same finite element spaces with IPDG method.  For the biharmonic equation, interior penalty discontinuous Galerkin finite element methods have been studied in \cite{dong,gh,ghv,msb,ms,sm}. One obvious disadvantage of discontinuous finite element methods is their rather complicated
formulations which are often necessary to guarantee well posedness and convergence of the methods.
 For example, the symmetric IPDG method for the biharmonic equation with homogenous boundary conditions \cite{dong,gh} has the following formulation:
\begin{eqnarray}
(\Delta u_h,\Delta v)_{\T_h}&+&\sum\int_e(\{\nabla\Delta u_h\}\cdot[v]+\{\nabla\Delta v\}\cdot[u_h])ds\nonumber\\
&+&\sum\int_e(\{\Delta u_h\}\cdot[\nabla v]+\{\Delta v\}\cdot[\nabla u_h])ds\nonumber\\
&+& \sum\int_e(\sigma [u_h]\cdot [v]+\tau[\nabla u_h][\nabla v])ds=(f,v), \label{ipdg}
\end{eqnarray}
where $\sigma$ and $\tau$ are two parameters that need to be tuned.

The purpose of this work is to introduce a conforming DG finite element method for the biharmonic equation which has the following ultra simple formulation without any stabilizing/penalty terms and other mixed terms of lower dimension integrations in (\ref{ipdg}):
\begin{equation}\label{cdg}
(\Delta_w u_h,\ \Delta_w v)=(f,\;v),
\end{equation}
where $\Delta_w$ is called weak Laplacian, an approximation of $\Delta$. The formulation (\ref{cdg}) can be viewed as a counterpart of (\ref{cfe}) for discontinuous approximations.
The conforming DG method was first introduced in \cite{cdg1,cdg2} for second order elliptic equations, which, by name, means the method using the finite element spaces of DG methods and the simple formulations of conforming methods. This new finite element method shares the same finite element space with the IPDG methods but having much simpler formulation. This simple formulation can be obtained by defining weak Laplacian $\Delta_w$ appropriately.  The idea here is to raise the degree of polynomials used to compute weak Laplacian $\Delta_w$.  Using higher degree polynomials in computation of weak Laplacian will not change the size, neither the global sparsity of the stiffness matrix. Optimal order error estimates in a discrete $H^2$ for $k\ge 2$ and in $L^2$ norm for $k>2$ are
established for the corresponding  finite element
solutions. Numerical results are provided to  confirm the theories.

\section{A Conforming DG Finite Element Method}\label{Section:wg-fem}

Let ${\mathcal T}_h$ be a partition of the domain $\Omega$ consisting of
polygons in two dimension or polyhedra in three dimension satisfying a set of conditions defined in \cite{wy-mixed} and additional conditions specified in \cite{sfwg-bi}.
Denote by ${\cal E}_h$ the set of all edges or flat faces in ${\cal
T}_h$, and let ${\cal E}_h^0={\cal E}_h\backslash\partial\Omega$ be
the set of all interior edges or flat faces.

For simplicity, we adopt the following notations,
\begin{eqnarray*}
(v,w)_{\T_h} &=& \sum_{T\in\T_h}(v,w)_T=\sum_{T\in\T_h}\int_T vw d\bx,\\
 \l v,w\r_{\partial\T_h}&=&\sum_{T\in\T_h} \l v,w\r_\pT=\sum_{T\in\T_h} \int_\pT vw ds.
\end{eqnarray*}
Let $P_k(K)$ consist all the polynomials degree less or equal to $k$ defined on $K$.

We define a finite element space $V_h$ for $k\ge 2$ as follows
\begin{equation}\label{vh}
V_h=\left\{ v\in L^2(\Omega):\ v|_{T}\in P_{k}(T)\;\; T\in\T_h \right\}.
\end{equation}

Let $T_1$ and $T_2$ be two polygons/polyhedrons
sharing $e$ if $e\in\E_h^0$.  Let $v$ and $\bq$ be  scalar and vector valued functions, the jumps $[v]$  and $[\bq]$ are defined as
\begin{equation}\label{jump}
[v]=v|_{T_1}\bn_1+ v|_{T_2}\bn_2, \quad  [\bq]=\bq|_{T_1}\cdot\bn_1+ \bq|_{T_2}\cdot\bn_2,
\end{equation}
and the averages $\{v\}$  and $\{\bq\}$ are defined as
\begin{equation}\label{avg}
\{v\}=\frac12(v|_{T_1}+v|_{T_2}) \quad \{\bq\}=\frac12 (\bq|_{T_1}+ \bq|_{T_2}).
\end{equation}
If $e$ is on $\partial\Omega$, then
\begin{equation}\label{avgb}
\{v\}=0,\quad \{\bq\}=0,\quad  [v]= v\bn,\quad [\bq]=\bq\cdot\bn.
\end{equation}

The new conforming DG finite element method for the biharmonic equation (\ref{pde})-(\ref{bc-n}) is defined as follows.

\begin{algorithm}
A numerical approximation for (\ref{pde})-(\ref{bc-n}) can be
obtained by seeking $u_h\in V_h$
satisfying the following equation:
\begin{equation}\label{wg}
(\Delta_w u_h,\ \Delta_w v)_{\T_h}=(f,\;v) \quad\forall v\in V_h.
\end{equation}
\end{algorithm}

Next we will discuss how to compute the weak Laplacian $\Delta_wu_h$ and $\Delta_w v$ in (\ref{wg}).  The concept of weak derivative was first introduced in \cite{wy,wy-mixed} for weak functions in weak Galerkin methods and was modified in \cite{mwg, mwg1}. A weak Laplacian operator, denoted by $\Delta_{w}$,
is defined as the unique polynomial $\Delta_{w}v \in P_j(T)$ for $j>k$ that
satisfies the following equation
\begin{equation}\label{wl}
(\Delta_{w} v, \ \varphi)_T = ( v, \ \Delta\varphi)_T-\l \{v\},\
\nabla\varphi\cdot\bn\r_\pT +\l \{\nabla v\}\cdot\bn, \ \varphi\r_\pT,\quad
\forall \varphi\in P_j(T).
\end{equation}


\begin{lemma}
Let $\phi\in H^2(\Omega)$, then on any $T\in\T_h$,
\begin{equation}\label{key}
\Delta_{w} \phi = \bbQ_h (\Delta \phi),
\end{equation}
where $\bbQ_h$ is a locally defined $L^2$ projections onto $P_{j}(T)$ on each element $T\in\T_h$.
\end{lemma}
\begin{proof}
It is not hard to see that for any $\tau\in P_j(T)$ we
have
\begin{eqnarray*}
(\Delta_{w} \phi,\ \tau)_T &=& (\phi,\ \Delta\tau)_T + \langle
\{\nabla \phi\}\cdot\bn, \;\tau\rangle_{\pT}-\langle  \{\phi\},\ \nabla\tau\cdot\bn \rangle_{\pT}\\
&=&(\phi, \Delta\tau)_T + \langle \nabla \phi\cdot\bn,\ \tau\rangle_{\partial
T}-\langle \phi,\ \nabla\tau\cdot\bn \rangle_{\pT}\\
&=&(\Delta \phi,\ \tau)_T=(\bbQ_h\Delta \phi,\ \tau)_T,
\end{eqnarray*}
which implies
\begin{equation}\label{key-Laplacian}
\Delta_{w}  \phi = \bbQ_h (\Delta \phi).
\end{equation}
We complete the proof.
\end{proof}
\section{Well Posedness}

First we define a semi-norm $\3bar\cdot\3bar$ as
\begin{equation}\label{3barnorm}
\3bar v\3bar^2=(\Delta_wv,\Delta_wv)_{\T_h}.
\end{equation}
Then we introduce a discrete $H^2$ norm as follows:
\begin{equation}\label{norm}
\|v\|_{2,h} = \left( \sum_{T\in\T_h}\left(\|\Delta v\|_T^2+h_T^{-3} \|  [v]\|^2_\pT+h_T^{-1}\|[\nabla v]\|^2_\pT\right) \right)^{\frac12}.
\end{equation}
The following lemma indicates that the two norms $\|\cdot\|_{2,h}$ and $\3bar\cdot\3bar$ are equivalent.

First we need the following trace inequality. For any function $\varphi\in H^1(T)$, the  trace
inequality holds true (see \cite{wy-mixed} for details):
\begin{equation}\label{trace}
\|\varphi\|_{e}^2 \leq C \left( h_T^{-1} \|\varphi\|_T^2 + h_T
\|\nabla \varphi\|_{T}^2\right).
\end{equation}

\begin{lemma} There exist two positive constants $C_1$ and $C_2$ such
that for any $v\in V_h$, we have
\begin{equation}\label{happy}
C_1 \|v\|_{2,h}\le \3bar v\3bar \leq C_2 \|v\|_{2,h}.
\end{equation}
\end{lemma}

\begin{proof}
For any $v\in V_h$, it follows from the definition of
weak Laplacian (\ref{wl}) and integration by parts that
\begin{eqnarray}
(\Delta_{w} v, \ \varphi)_T &=& ( v, \ \Delta\varphi)_T-\l \{v\},\
\nabla\varphi\cdot\bn\r_\pT +\l \{\nabla v\}\cdot\bn, \ \varphi\r_\pT\nonumber\\
&=&-(\nabla v, \ \nabla\varphi)_T+\l v-\{v\},\
\nabla\varphi\cdot\bn\r_\pT +\l \{\nabla v\}\cdot\bn, \ \varphi\r_\pT\nonumber\\
&=&(\Delta v, \ \varphi)_T+\l v-\{v\},\
\nabla\varphi\cdot\bn\r_\pT +\l (\{\nabla v\}-\nabla v)\cdot\bn, \ \varphi\r_\pT.\label{n-1}
\end{eqnarray}
By letting $\varphi=\Delta_w v$ in (\ref{n-1}) we arrive at
\begin{eqnarray*}
\|\Delta_{w} v\|^2_T &=&(\Delta v, \ \Delta_w v)_T+\l v-\{v\},\
\nabla(\Delta_w v)\cdot\bn\r_\pT +\l (\{\nabla v\}-\nabla v)\cdot\bn, \ \Delta_w v\r_\pT
\end{eqnarray*}
It is easy to see that the following equations hold true for $v$  on $T$ with $e\subset\pT$,
\begin{equation}\label{jp}
\|v-\{v\}\|_e=\|[v]\|_e\quad {\rm if} \;e\subset \partial\Omega,\quad\|v-\{v\}\|_e=\frac12\|[v]\|_e\;\; {\rm if} \;e\in\E_h^0.
\end{equation}
and
\begin{equation}\label{jp1}
\begin{split}
&\|(\nabla v-\{\nabla v\})\cdot\bn\|_e=\|[\nabla v]\|_e\quad {\rm if} \;e\subset \partial\Omega,\\
&|(\nabla v-\{\nabla v\})\cdot\bn\|_e=\frac12\|[\nabla v]\|_e\;\; {\rm if} \;e\in\E_h^0.
\end{split}
\end{equation}
From the trace inequality (\ref{trace}), (\ref{jp})-(\ref{jp1}) and the inverse inequality
we have
\begin{eqnarray*}
\|\Delta_wv\|^2_T &\le& \|\Delta v\|_T \|\Delta_w v\|_T+ \|v-\{v\}\|_\pT \|\nabla(\Delta_w v)\|_\pT+\|(\{\nabla v\}-\nabla v)\cdot\bn\|_\pT \|\Delta_w v\|_\pT\\
&\le& C(\|\Delta v\|_T + h_T^{-3/2}\|[v]\|_\pT+h_T^{-1/2}\|[\nabla v]\|_\pT) \|\Delta_w v\|_T,
\end{eqnarray*}
which implies
$$
\|\Delta_w v\|_T \le C \left(\|\Delta v\|_T +h_T^{-3/2}\|[v]\|_\pT+h_T^{-1/2}\|[ \nabla v]\|_\pT\right),
$$
and consequently
$$\3bar v\3bar \leq C_2 \|v\|_{2,h}.$$
Next we will prove
$$
\sum h_T^{-3}\|[v]\|^2_\pT\le C\3bar v\3bar.
$$
It follows from (\ref{n-1}) that for any $\varphi\in P_j(T)$,
\begin{eqnarray}
(\Delta_{w} v, \ \varphi)_T
&=&(\Delta v, \ \varphi)_T+\l v-\{v\},\
\nabla\varphi\cdot\bn\r_\pT\nonumber\\
& +&\l (\{\nabla v\}-\nabla v)\cdot\bn, \ \varphi\r_\pT.\label{n-11}
\end{eqnarray}
By Lemma 3.1 in \cite{sfwg-bi}, there exist a $\varphi_0$ such that for $e\subset\pT$,
\begin{eqnarray*}
&&(\Delta v,\varphi_0)_T=0,\;\;\l (\{\nabla v\}-\nabla v)\cdot\bn, \ \varphi_0\r_\pT=0,\nonumber\\
&& \l v-\{v\},\nabla\varphi_0\cdot\bn\r_{\pT\setminus e}=0,\;\; \l v-\{v\},\nabla\varphi_0\cdot\bn\r_\pT=\|v-\{v\}\|_e^2.
\end{eqnarray*}
and
\begin{eqnarray*}
\|\varphi_0\|_T\le C h_T^{3/2}\|v-\{v\}\|_e.
\end{eqnarray*}
Letting $\varphi=\varphi_0$ in (\ref{n-11}) yields
\begin{eqnarray*}
\|v-\{v\}\|_e^2&=&(\Delta_{w} v, \ \varphi_0)_T\le \|\Delta_{w} v\|_T\|\varphi_0\|_T\le Ch_T^{3/2} \|\Delta_{w} v\|_T\|v-\{v\}\|_e.
\end{eqnarray*}
which implies
\begin{eqnarray*}
\|v-\{v\}\|_e\le C h_T^{3/2} \|\Delta_{w} v\|_T.
\end{eqnarray*}
Taking the summation of the above equation over $T\in\T_h$ and using (\ref{jp}), one has
\begin{equation}\label{n-22}
\sum_{T\in\T_h} h_T^{-3}\|[v]\|^2_\pT\le C\3bar v\3bar^2.
\end{equation}
Similarly, by Lemma 3.2 in \cite{sfwg-bi}, we can have
\begin{equation}\label{n-33}
\sum_{T\in\T_h} h_T^{-1}\|[\nabla v]\|^2_\pT \le C\3bar v\3bar^2.
\end{equation}
Finally, by letting $\varphi=\Delta v$ in (\ref{n-11}) we arrive at
\begin{eqnarray*}
\|\Delta v\|^2_T &=&(\Delta v, \ \Delta_w v)_T-\l v-\{v\},\
\nabla(\Delta_w v)\cdot\bn\r_\pT \\
   &&\quad \ -  \l (\{\nabla v\}-\nabla v)\cdot\bn, \ \Delta_w v\r_\pT.
\end{eqnarray*}
Using the trace inequality (\ref{trace}), the inverse inequality and (\ref{n-22})-(\ref{n-33}), one has
\begin{eqnarray*}
\|\Delta v\|^2_T &\le& C \|\Delta_w v\|_T\|\Delta v\|_T,
\end{eqnarray*}
which gives
\begin{eqnarray*}
\sum_{T\in\T_h}\|\Delta v\|^2_T &\le& C \3bar v\3bar^2.
\end{eqnarray*}
We complete the proof.

\end{proof}

\smallskip

\begin{lemma}
The finite element scheme (\ref{wg}) has a unique
solution.
\end{lemma}

\begin{proof}
It suffices to show that the solution of (\ref{wg}) is trivial if
$f=g=\phi=0$. It follows that
\[
(\Delta_w u_h,\Delta_w u_h)_{\T_h}=0.
\]
Then the norm equivalence (\ref{happy}) implies $\|u_h\|_{2,h}=0$, i.e.
$$\sum_{T\in\T_h} \|\Delta u_h\|_T^2=0,\;\; \sum_{T\in\T_h} h_T^{-3}\|[u_h]\|=0,\;\ \sum_{T\in\T_h} h_T^{-1}\|[\nabla u_h]\|=0.$$
Therefore,  $u_h$ is a smooth harmonic function  on $\Omega$ and $u_h=0$ on $\partial\Omega$. Thus  we have $u_h=0$, which completes the proof.
\end{proof}

\section{An Error Equation}

Let $e_h=u-u_h$. Next we derive an error equation that $e_h$ satisfies.

\begin{lemma}
For any $v\in V_h$, we have
\begin{eqnarray}
(\Delta_we_h,\Delta_wv)_{\T_h}=\ell_1(u,v)+\ell_2(u,v),\label{ee}
\end{eqnarray}
where
\begin{eqnarray*}
\ell_1(u,v)&=& \langle \nabla(\bbQ_h\Delta u-\Delta u)\cdot\bn, v-\{v\}\rangle_{\pT_h},\\
\ell_2(u,v)&=& \langle \Delta u-\bbQ_h\Delta  u, (\nabla v-\{\nabla v\})\cdot\bn\rangle_{\pT}.
\end{eqnarray*}
\end{lemma}

\begin{proof}
Testing (\ref{pde}) by  $v\in V_h$  and using the fact that
$\sum_{T\in\T_h}\langle \nabla (\Delta u)\cdot\bn, \{v\}\rangle_\pT=0$ and $\sum_{T\in\T_h}\langle \Delta u, \{\nabla v\}\cdot\bn\rangle_\pT=0$ and integration by parts,  we arrive at
\begin{eqnarray}
(f,v)&=&(\Delta^2u, v)_{\T_h}\nonumber\\
&=&(\Delta u,\Delta v)_{\T_h} -\langle \Delta u,
\nabla v\cdot\bn\rangle_{\pT_h} + \langle\nabla(\Delta u)\cdot\bn,
v\rangle_{\pT_h}\label{m1}\\
&=&(\Delta u,\Delta v)_{\T_h} -\langle \Delta u,
(\nabla v-\{\nabla v\})\cdot\bn\rangle_{\pT_h}\nonumber\\
& +& \langle\nabla(\Delta u)\cdot\bn,
v-\{v\}\rangle_{\pT_h}.\nonumber
\end{eqnarray}
Next we investigate the term  $(\Delta u,\Delta v)_{\T_h}$ in the above equation. Using (\ref{key}), integration by parts and the definition of weak Laplacian (\ref{wl}), we have
\begin{eqnarray*}
& & (\Delta u, \Delta v)_{\T_h}=(\bbQ_h\Delta u, \Delta v)_{\T_h} \\
&=& (v, \Delta(\bbQ_h\Delta u))_T + \langle \nabla v\cdot\bn,\ \bbQ_h\Delta u\rangle_{\pT}-\langle v, \nabla(\bbQ_h\Delta u)\cdot\bn \rangle_{\pT}\nonumber\\
&=&(\Delta_w v,\ \bbQ_h\Delta u)_T-\langle v-\{v\}, \nabla(\bbQ_h\Delta u)\cdot\bn\rangle_{\pT}+\langle (\nabla v-\{\nabla v\})\cdot\bn, \bbQ_h\Delta
u\rangle_{\pT}\nonumber\\
&=&(\Delta_w u,\ \Delta_w v)_T- \langle v-\{v\}, \nabla(\bbQ_h\Delta u)\cdot\bn\rangle_{\pT}+\langle (\nabla v-\{\nabla v\})\cdot\bn, \bbQ_h\Delta
u\rangle_{\pT}.
\end{eqnarray*}
Combining the above two equations  gives
\begin{eqnarray}
(f,v)&=&(\Delta^2u, v)_{\T_h}\nonumber\\
&=&(\Delta_w u,\ \Delta_w v)_{\T_h}-\langle v-\{v\}, \nabla(\bbQ_h\Delta u-\Delta u)\cdot\bn\rangle_{\pT_h}\nonumber\\
&-&\langle (\nabla v-\{\nabla v\})\cdot\bn, \Delta u-\bbQ_h\Delta u \rangle_{\pT},\label{mmmm}
\end{eqnarray}
which implies that
\begin{eqnarray*}
(\Delta_w u,\ \Delta_w v)_{\T_h}=(f,v)+\ell_1(u,v)+\ell_2(u,v).
\end{eqnarray*}
The error equation follows from subtracting (\ref{wg}) from the above equation,
\begin{eqnarray*}
(\Delta_w e_h,\ \Delta_w v)_{\T_h}=\ell_1(u,v)+\ell_2(u,v).
\end{eqnarray*}
We have proved the lemma.
\end{proof}

\section{An Error Estimate in $H^2$}
We start this section by defining some approximation operator. Let $Q _h$ be the element-wise defined $L^2$ projection onto $P_{k}(T)$ on each element $T$.

\begin{lemma}\label{l2}
Let $k\ge 2$ and $w\in H^{\max\{k+1,4\}}(\Omega)$. There exists a constant $C$ such that the following estimates hold true:
\begin{eqnarray}
&  \left(\sum_{T\in\T_h} h_T\|\Delta w-\bbQ_h\Delta w\|_{\partial
T}^2\right)^{\frac12}
\leq C h^{k-1}\|w\|_{k+1},\label{mmm1}\\
&  \left(\sum_{T\in\T_h} h_T^3\|\nabla(\Delta w-\bbQ_h\Delta
w)\|_\pT^2\right)^{\frac12} \leq Ch^{k-1}(\|w\|_{k+1}
+h\delta_{k,2}\|w\|_4).
\label{mmm2}
\end{eqnarray}
Here $\delta_{i,j}$ is the usual Kronecker's delta with value $1$
when $i=j$ and value $0$ otherwise.
\end{lemma}

The above lemma can be proved by using the trace inequality (\ref{trace}) and the definition of $\bbQ_h$. The proof can also be found in \cite{mwy-bi}.

\begin{lemma} Let  $w\in H^{\max\{k+1,4\}}(\Omega)$, and $v\in V_h$.
There exists a constant $C$ such that
\begin{eqnarray}
|\ell_1(w, v)|&\le&  C h^{k-1}(\|w\|_{k+1} + h\delta_{k,2} \|w\|_4)\3bar v\3bar.\label{mm1}\\
|\ell_2(w, v)|&\le& Ch^{k-1}|w|_{k+1}\3bar v\3bar.\label{mm2}
\end{eqnarray}
\end{lemma}

\begin{proof}
Using the Cauchy-Schwartz inequality, (\ref{mmm1}), (\ref{mmm2}), (\ref{jp}), (\ref{jp1}) and (\ref{happy}), we have
\begin{eqnarray}
\ell_1(w,v)&=&\left|\sum_{T\in\T_h}\langle \nabla(\Delta w-\bbQ_h\Delta
w)\cdot\bn, v-\{v\}\rangle_\pT\right| \nonumber\\
&\le&  \left(\sum_{T\in\T_h}h_T^3\|\nabla(\Delta w-\bbQ_h\Delta
w)\|_\pT^2\right)^{\frac12}
\left(\sum_{T\in\T_h}h_T^{-3}\|v-\{v\}\|^2_{\pT}\right)^{\frac12}\nonumber\\
&\le&  \left(\sum_{T\in\T_h}h_T^3\|\nabla(\Delta w-\bbQ_h\Delta
w)\|_\pT^2\right)^{\frac12}
\left(\sum_{T\in\T_h}h_T^{-3}\|[v]\|^2_{\pT}\right)^{\frac12}\nonumber\\
&\le& C h^{k-1}(\|w\|_{k+1} + h\delta_{k,2} \|w\|_4) \3barv\3bar,\label{ell-1}
\end{eqnarray}
and
\begin{eqnarray}
\ell_2(w,v)&=&\left|\sum_{T\in\T_h} \langle \Delta w-\bbQ_h\Delta w, (\nabla v-\{\nabla v\})\cdot\bn\rangle_\pT\right|\nonumber\\
\le &&\left(\sum_{T\in\T_h} h_T\|\Delta w-\bbQ_h\Delta
w\|_\pT^2\right)^{\frac12} \left(\sum_{T\in\T_h} h_T^{-1}
\|[\nabla v]\|_\pT^2\right)^{\frac12}\nonumber\\
\le && C h^{k-1}\|w\|_{k+1} \3bar v\3bar.\label{ell-2}
\end{eqnarray}
We complete the proof.
\end{proof}

\begin{lemma}
Let  $w\in H^{\max\{k+1,4\}}(\Omega)$,  then
\begin{equation}\label{eee2}
\3bar w-Q_hw\3bar\le Ch^{k-1}|w|_{k+1}.
\end{equation}
\end{lemma}
\begin{proof}
For any $T\in\T_h$, it follows from (\ref{wl}), integration  by parts, (\ref{trace}) and inverse inequality that for $w\in P_j(T)$,
\begin{eqnarray*}
\|\Delta_w(w-Q_hw)\|_T^2&=&(\Delta_w(w-Q_hw), \Delta_w(w-Q_hw))_{T}\\
 &=&(w-Q_hw, \Delta (\Delta_w(w-Q_hw)))_{T}-\l \{w-Q_hw\}, \nabla (\Delta_w(w-Q_hw))\cdot\bn\r_\pT\\
&+&\l \{\nabla w-\nabla Q_hw\}\cdot\bn, \Delta_w(w-Q_hw)\r_{\pT}\\
&\le&C(h_T^{-2}\|w-Q_hw\|_T+h_T^{-3/2}\|w-Q_hw\|_\pT\\
&+&h_T^{-1/2}\|\nabla w-\nabla Q_hw\|_\pT)\|\Delta_w(w-Q_hw)\|_T\\
&\le& Ch^{k-1}|w|_{k+1, T}\|\Delta_w(w-Q_hw)\|_T,
\end{eqnarray*}
which implies
\begin{eqnarray*}
\|\Delta_w(w-Q_hw)\|_T &\le& Ch^{k-1}|w|_{k+1, T}.
\end{eqnarray*}
Taking the summation over $T\in\T_h$, we have proved the lemma.
\end{proof}

\begin{theorem} Let $u_h\in V_h$  be the finite element solution arising from
(\ref{wg}). Assume that the exact solution $u\in H^{\max\{k+1,4\}}(\Omega)$. Then, there
exists a constant $C$ such that
\begin{equation}\label{err1}
\3bar u-u_h\3bar \le Ch^{k-1}\left(\|u\|_{k+1}+h\delta_{k,2}\|u\|_{4}\right).
\end{equation}
\end{theorem}
\begin{proof}
Let $\epsilon_h=Q_hu-u_h$. Then it is straightforward to obtain
\begin{eqnarray}
\3bar e_h\3bar^2&=&(\Delta_we_h, \Delta_we_h)_{\T_h}\label{eee1}\\
&=&(\Delta_we_h, \Delta_w(u-u_h))_{\T_h}\nonumber\\
&=&(\Delta_we_h,\Delta_w(Q_hu-u_h))_{\T_h}+(\Delta_we_h, \Delta_w(u-Q_hu))_{\T_h}\nonumber\\
&=&(\Delta_we_h,\Delta_w\epsilon_h)_{\T_h}+(\Delta_we_h, \Delta_w(u-Q_hu))_{\T_h}.\nonumber
\end{eqnarray}
We will bound the term  $(\Delta_we_h,\Delta_w\epsilon_h)_{\T_h}$ on right hand side of (\ref{eee1}) first.
Letting $v=\epsilon_h\in V_h$ in (\ref{ee})  and using (\ref{mm1})-(\ref{mm2}) and (\ref{eee2}), we have
\begin{eqnarray}
|(\Delta_we_h,\Delta_w\epsilon_h)_{\T_h}|&\le&|\ell_1(u,\epsilon_h)|+|\ell_2(u,\epsilon_h)|\nonumber\\
&\le& Ch^{k-1}(\|u\|_{k+1}+h\delta_{k,2}\|u\|_{4})\3bar \epsilon_h\3bar\nonumber\\
&\le& Ch^{k-1}(\|u\|_{k+1}+h\delta_{k,2}\|u\|_{4})(\3bar u-Q_hu\3bar+\3bar u-u_h\3bar)\nonumber\\
&\le& Ch^{2(k-1)}(\|u\|^2_{k+1}+h^2\delta_{k,2}\|u\|^2_{4})+\frac14 \3bare_h\3bar^2.\label{eee3}
\end{eqnarray}
To bound the second term on right hand side of (\ref{eee1}), we have by (\ref{eee2}),
\begin{eqnarray}
|(\Delta_we_h, \Delta_w(u-Q_hu))_{\T_h}|&\le& C\3bar u-Q_hu\3bar \3bar e_h\3bar\nonumber\\
&\le& Ch^{2(k-1)}\|u\|^2_{k+1}+\frac14\3bar e_h\3bar^2.\label{eee4}
\end{eqnarray}
Combining the estimates (\ref{eee3}) and  (\ref{eee4}) with (\ref{eee1}), we arrive
\[
\3bar e_h\3bar \le Ch^{k-1}\left(\|u\|_{k+1}+h\delta_{k,2}\|u\|_{4}\right),
\]
which completes the proof.
\end{proof}

\section{Error Estimates in $L^2$ Norm}

In this section, we will obtain an error bound for the finite element solution $u_h$ in
$L^2$ norm.


The dual problem considered has the following form,
\begin{eqnarray}
\Delta^2w&=& e_h\quad
\mbox{in}\;\Omega,\label{dual}\\
w&=&0\quad\mbox{on}\;\partial\Omega,\label{dual1}\\
\nabla w\cdot\bn&=&0\quad\mbox{on}\;\partial\Omega.\label{dual2}
\end{eqnarray}
Assume that the $H^{4}$ regularity holds,
\begin{equation}\label{reg}
\|w\|_4\le C\|e_h\|.
\end{equation}

\begin{theorem}
Let $u_h\in V_h$ be the finite element solution
arising from (\ref{wg}). Assume that the exact solution $u\in H^{k+1}(\Omega)$ and (\ref{reg}) holds true.
 Then, there exists a constant $C$ such that
\begin{equation}\label{err2}
\|u-u_h\| \le Ch^{k+1-\delta_{k,2}}(\|u\|_{k+1}+ h\delta_{k,2}\|u\|_4).
\end{equation}
\end{theorem}
\begin{proof}
Testing (\ref{dual}) by  $e_h$  and using the fact that
$\sum_{T\in\T_h}\langle \nabla (\Delta w)\cdot\bn, \{e_h\}\rangle_\pT=0$ and $\sum_{T\in\T_h}\langle \Delta w, \{\nabla e_h\}\cdot\bn\rangle_\pT=0$ and integration by parts,  we arrive at
\begin{eqnarray*}
\|e_h\|^2&=&(\Delta^2 w, e_h)_{\T_h}\nonumber\\
&=&(\Delta w,\Delta e_h)_{\T_h} -\langle \Delta w,
\nabla e_h\cdot\bn\rangle_{\pT_h} + \langle\nabla(\Delta w)\cdot\bn,
e_h\rangle_{\pT_h}\nonumber\\
&=&(\Delta w,\Delta e_h)_{\T_h} -\langle \Delta w,
(\nabla e_h-\{\nabla e_h\})\cdot\bn\rangle_{\pT_h}\nonumber\\
& +& \langle\nabla(\Delta w)\cdot\bn,
e_h-\{e_h\}\rangle_{\pT_h}.\nonumber\\
&=&(\bbQ_h\Delta w,\Delta e_h)_{\T_h}+(\Delta w-\bbQ_h\Delta w,\Delta e_h)_{\T_h}\\
& -&\langle \Delta w,
(\nabla e_h-\{\nabla e_h\})\cdot\bn\rangle_{\pT_h}
+ \langle\nabla(\Delta w)\cdot\bn,
e_h-\{e_h\}\rangle_{\pT_h}.\nonumber
\end{eqnarray*}
It follows from integration by parts, the definition of weak Laplacian (\ref{wl}) and (\ref{key}),
\begin{eqnarray*}
(\bbQ_h\Delta w, \Delta e_h)_{\T_h}
&=& (e_h, \Delta(\bbQ_h\Delta w))_T + \langle \nabla e_h\cdot\bn,\ \bbQ_h\Delta w\rangle_{\pT}-\langle e_h, \nabla(\bbQ_h\Delta w)\cdot\bn \rangle_{\pT}\nonumber\\
&=&(\Delta_w e_h,\ \bbQ_h\Delta w)_T-\langle e_h-\{e_h\}, \nabla(\bbQ_h\Delta w)\cdot\bn\rangle_{\pT}\\
&+&\langle (\nabla e_h-\{\nabla e_h\})\cdot\bn, \bbQ_h\Delta
w\rangle_{\pT}\\
&=&(\Delta_w w,\ \Delta_w e_h)_T- \langle e_h-\{e_h\}, \nabla(\bbQ_h\Delta w)\cdot\bn\rangle_{\pT}\\
&+&\langle (\nabla e_h-\{\nabla e_h\})\cdot\bn, \bbQ_h\Delta
w\rangle_{\pT}.
\end{eqnarray*}
Combining the two equations above implies
\begin{eqnarray*}
\|e_h\|^2&=&(\Delta_w w,\ \Delta_w e_h)_{\T_h}+(\Delta w-\bbQ_h\Delta w,\Delta e_h)_{\T_h}\\
&+&\ell_1(w,e_h)+\ell_2(w,e_h).
\end{eqnarray*}
By simple manipulation and (\ref{ee}), we have
\begin{eqnarray*}
(\Delta_w w,\ \Delta_w e_h)_{\T_h}&=&(\Delta_w Q_hw,\ \Delta_w e_h)_{\T_h}+(\Delta_w (w-Q_h),\ \Delta_w e_h)_{\T_h}\\
&=&\ell_1(u,Q_hw)+\ell_2(u, Q_hw)+(\Delta_w (w-Q_h),\ \Delta_w e_h)_{\T_h}.
\end{eqnarray*}
Combining the two equations above implies
\begin{eqnarray*}
\|e_h\|^2&=&\ell_1(u,Q_hw)+\ell_2(u, Q_hw)+(\Delta_w e_h,\ \Delta_w (w-Q_hw))_{\T_h}\\
&+&(\Delta w-\bbQ_h\Delta w,\Delta e_h)_{\T_h}
+\ell_1(w,\epsilon_h)+\ell_2(w,\epsilon_h)\\
&=&I_1+I_2+I_3+I_4+I_5+I_6.
\end{eqnarray*}

Next, we will estimate all the terms on the right hand side of the above equation.
Using the Cauchy-Schwartz inequality, (\ref{mmm1})-(\ref{mmm2}), (\ref{jp}) and (\ref{ell-1}), we have
\begin{eqnarray*}
I_1&=&\ell_1(u,Q_hw)=\left|\sum_{T\in\T_h}\langle \nabla(\Delta u-\bbQ_h\Delta u)\cdot\bn, Q_hw-\{Q_hw\}\rangle_\pT\right| \\
\le && \left(\sum_{T\in\T_h}h^3_T\|\nabla(\Delta u-\bbQ_h\Delta
u)\|_\pT^2\right)^{\frac12}
\left(\sum_{T\in\T_h}h_T^{-3}\|[Q_hw]\|^2_{\pT}\right)^{\frac12}\\
\le && \left(\sum_{T\in\T_h}h^3_T\|\nabla(\Delta u-\bbQ_h\Delta
u)\|_\pT^2\right)^{\frac12}
\left(\sum_{T\in\T_h}h_T^{-3}\|[Q_hw-w]\|^2_{\pT}\right)^{\frac12}\\
\le &&  C h^{k+1-\delta_{k,2}}\left(\|u\|_{k+1}+h\delta_{k,2}\|u\|_{4}\right)\|w\|_4,
\end{eqnarray*}
Similarly, by the Cauchy-Schwartz inequality, (\ref{mmm1})-(\ref{mmm2}), (\ref{jp1}) and (\ref{ell-2}), we have
\begin{eqnarray*}
I_2&=&\ell_2(u,Q_hw)=\left|\sum_{T\in\T_h} \langle \Delta u-\bbQ_h\Delta u, (\nabla Q_hw-\{\nabla Q_hw\})\cdot\bn)\rangle_\pT\right|\nonumber\\
\le &&\left(\sum_{T\in\T_h} h^{-1}_T\|\Delta u-\bbQ_h\Delta
u\|_\pT^2\right)^{\frac12} \left(\sum_{T\in\T_h} h_T
\|[\nabla Q_hw-\nabla w]\|_\pT^2\right)^{\frac12}\nonumber\\
\le &&  C h^{k+1-\delta_{k,2}}\|u\|_{k+2} \|w\|_4.
\end{eqnarray*}
The estimates (\ref{err1}) and (\ref{eee2}) give
\begin{eqnarray*}
I_3\le  C h^{k+1}\|u\|_{k+1} \|w\|_4.
\end{eqnarray*}

To estimate $I_4$, we need to bound $\|\Delta e_h\|_T$.
By (\ref{happy}), (\ref{eee2}), (\ref{err1}) and the definition of $Q_h$, we have
\begin{eqnarray*}
\sum_{T\in\T_h}\|\Delta e_h\|^2_{T}&\le& \sum_{T\in\T_h}\|\Delta\epsilon_h\|^2_{T}+\sum_{T\in\T_h}\|\Delta(u-Q_hu)\|^2_{T}\\
&\le&C(h^{k-1}\|u\|_{k+1}+\3bar \epsilon_h\3bar^2)\\
&\le&C(h^{k-1}\|u\|_{k+1}+\3bar e_h\3bar^2+\3bar Q_hu-u\3bar^2)\\
&\le&Ch^{k-1}\|u\|_{k+1}.
\end{eqnarray*}

The above estimate and the definition of $\bbQ_h$ imply
\begin{eqnarray*}
I_4 &\le& C h^{k+1}\|u\|_{k+1} \|w\|_4.
\end{eqnarray*}
Using the Cauchy-Schwartz inequality, (\ref{jp}), (\ref{mmm1}), (\ref{happy}), (\ref{eee2}) and (\ref{err1}), we have
\begin{eqnarray*}
I_5&=&\ell_1(w,e_h)=\left|\sum_{T\in\T_h}\langle \nabla(\Delta w-\bbQ_h\Delta
w)\cdot\bn, e_h-\{e_h\}\rangle_\pT\right|\\
&\le&  \left(\sum_{T\in\T_h}h_T^3\|\nabla(\Delta w-\bbQ_h\Delta
w)\|_\pT^2\right)^{\frac12}
\left(\sum_{T\in\T_h}h_T^{-3}\|e_h-\{e_h\}\|^2_{\pT}\right)^{\frac12}\\
&\le&  \left(\sum_{T\in\T_h}h_T^3\|\nabla(\Delta w-\bbQ_h\Delta
w)\|_\pT^2\right)^{\frac12}
\left(\sum_{T\in\T_h}h_T^{-3}\|[e_h]\|^2_{\pT}\right)^{\frac12}\\
&\le& C h^2\|w\|_4\left(\3bar \epsilon_h\3bar+(\sum_{T\in\T_h}h_T^{-3}\|[Q_hu-u]\|^2_{\pT})^{\frac12}\right)\\
&\le& C h^2\|w\|_4\left(\3bar e_h\3bar+\3bar Q_hu-u\3bar+(\sum_{T\in\T_h}h_T^{-3}\|[Q_hu-u]\|^2_{\pT})^{\frac12}\right)\\
&\le&C h^{k+1}\|u\|_{k+1} \|w\|_4.
\end{eqnarray*}
Similarly, we obtain
\begin{eqnarray*}
I_6\le C h^{k+1}\|u\|_{k+1} \|w\|_4.
\end{eqnarray*}
Combining all the estimates above yields
$$
\|e_h\|^2 \leq C h^{k+1-\delta_{k,2}}\|u\|_{k+1} \|w\|_4.
$$
It follows from the above inequality and
the regularity assumption (\ref{reg}).
 $$
\|e_h\|\leq C h^{k+1-\delta_{k,2}}\|u\|_{k+1}.
$$
We have completed the proof.
\end{proof}

\section{Numerical Experiments}\label{Section:numerical-results}

\long\def\a#1{\begin{align*}#1\end{align*}}\def\b#1{{\mathbf #1}}

 We solve the following biharmonic equation by $P_k$ conforming DG finite element methods:
\begin{align} \label{s1}  \Delta^2 u =f  ,  \quad (x,y)\in\Omega=(0,1)^2,
\end{align} with the boundary conditions $u=g_1$ and $\nabla u\cdot \b n=g_2$ on $\partial \Omega$.
We choose $f$, $g_1$ and $g_2$  so that the exact solution is
\a{ u=e^{x+y}. }

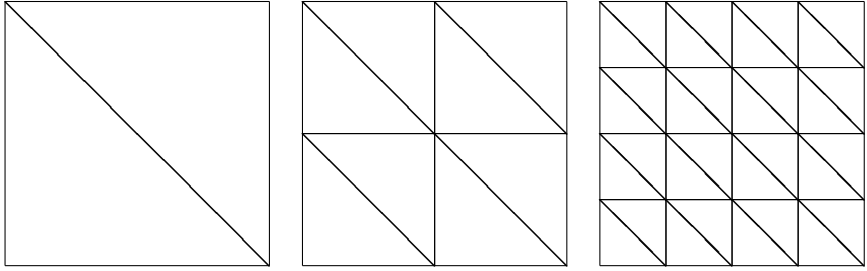
\begin{figure}[h!]
 \begin{center} \setlength\unitlength{1.25pt}
\begin{picture}(260,80)(0,0)
  \def\tr{\begin{picture}(20,20)(0,0)\put(0,0){\line(1,0){20}}\put(0,20){\line(1,0){20}}
          \put(0,0){\line(0,1){20}} \put(20,0){\line(0,1){20}}
   \put(0,20){\line(1,-1){20}}  \put(20,0){\line(-1,1){20}}\end{picture}}
 {\setlength\unitlength{5pt}
 \multiput(0,0)(20,0){1}{\multiput(0,0)(0,20){1}{\tr}}}

  {\setlength\unitlength{2.5pt}
 \multiput(45,0)(20,0){2}{\multiput(0,0)(0,20){2}{\tr}}}

  \multiput(180,0)(20,0){4}{\multiput(0,0)(0,20){4}{\tr}}

 \end{picture}\end{center}\label{grid1}
\caption{The first three levels of grids used in the computation of Table \ref{t1}. }
\end{figure}

In the first computation, the first three levels of grids are plotted in Figure \ref{grid1}.
The error and the order of convergence for the  method are listed in Tables \ref{t1}.
Here on triangular grids,  we compute the  weak
  Laplacian $\Delta_w v$ by $P_{k+2}$ polynomials.
The numerical results confirm the convergence theory.

\begin{table}[h!]
  \centering \renewcommand{\arraystretch}{1.1}
  \caption{The error and the order of convergence for \eqref{s1} on triangular grids (Figure \ref{grid1}) }\label{t1}
\begin{tabular}{c|cc|cc|cc}
\hline
level & $\|u_h-  u\|_0 $  &rate & $ |u_h-u|_{1,h} $ &rate  & $\3bar u_h- u\3bar $ &rate    \\
\hline
 &\multicolumn{6}{c}{by the $P_2$ conforming DG finite element } \\ \hline
 4&   0.3653E-03 &  2.0&   0.3281E-02 &  1.9&   0.1229E+01 &  0.9 \\
 5&   0.9566E-04 &  1.9&   0.8733E-03 &  1.9&   0.6312E+00 &  1.0 \\
 6&   0.2480E-04 &  1.9&   0.2268E-03 &  1.9&   0.3199E+00 &  1.0 \\
 \hline
 &\multicolumn{6}{c}{by the $P_3$ conforming DG finite element } \\ \hline
 2&   0.2291E-03 &  4.4&   0.3275E-02 &  3.1&   0.1612E+00 &  2.0 \\
 3&   0.1143E-04 &  4.3&   0.3889E-03 &  3.1&   0.4577E-01 &  1.8 \\
 4&   0.7148E-06 &  4.0&   0.4743E-04 &  3.0&   0.1243E-01 &  1.9 \\ 
 \hline
\end{tabular}%
\end{table}%

In the next computation,  we use a family of polygonal grids (with pentagons and of 8-side polygons)
   shown in Figure \ref{5g}.
We let the polynomial degree $j=k+3$ for the weak Laplacian on such polygonal meshes.
The rate of convergence is listed in Table \ref{t2}.
The convergence history confirms the theory.

\begin{figure}[htb]\begin{center}\setlength\unitlength{1.5in}
    \begin{picture}(3.2,1.4)
 \put(0,0){\includegraphics[width=1.5in]{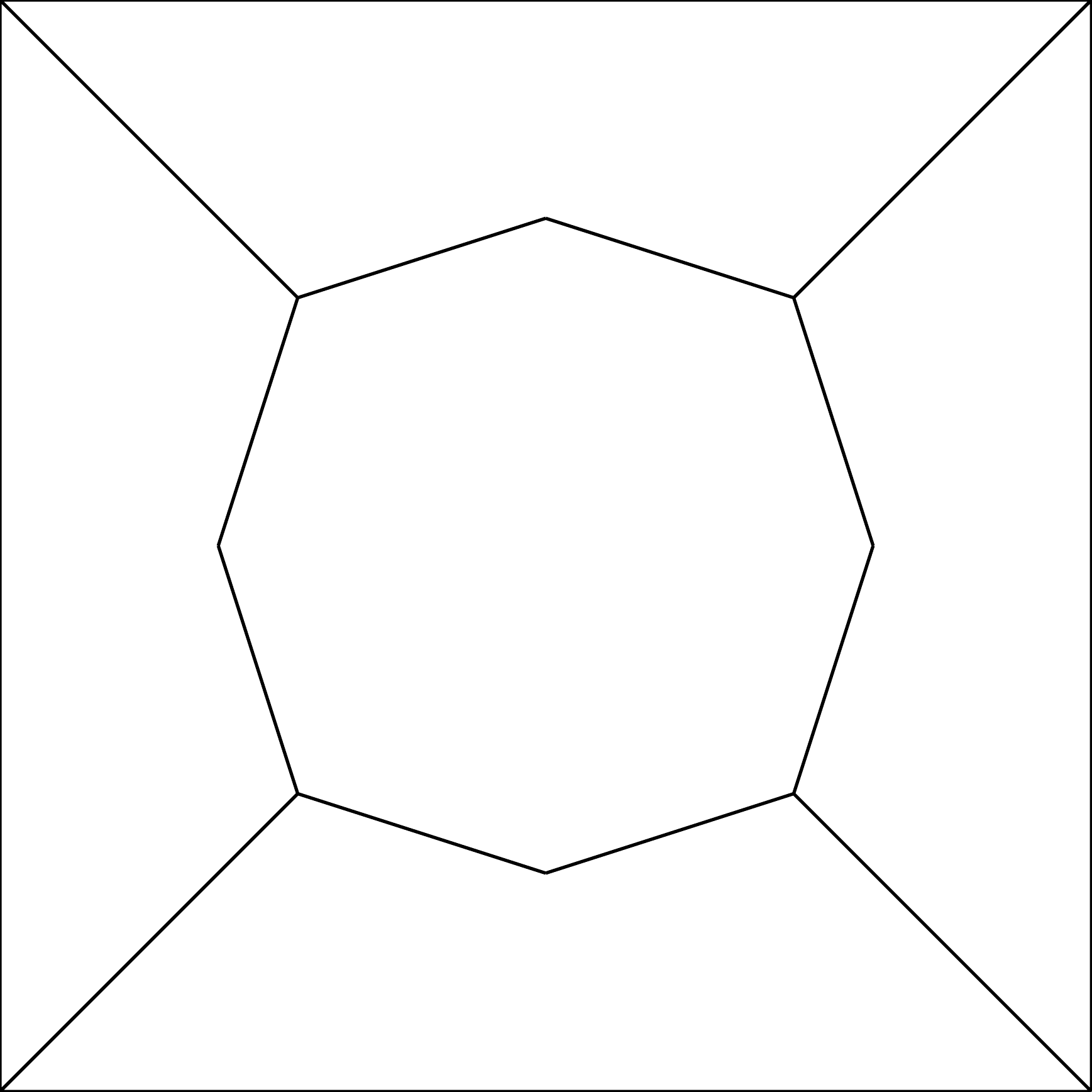}}  
 \put(1.1,0){\includegraphics[width=1.5in]{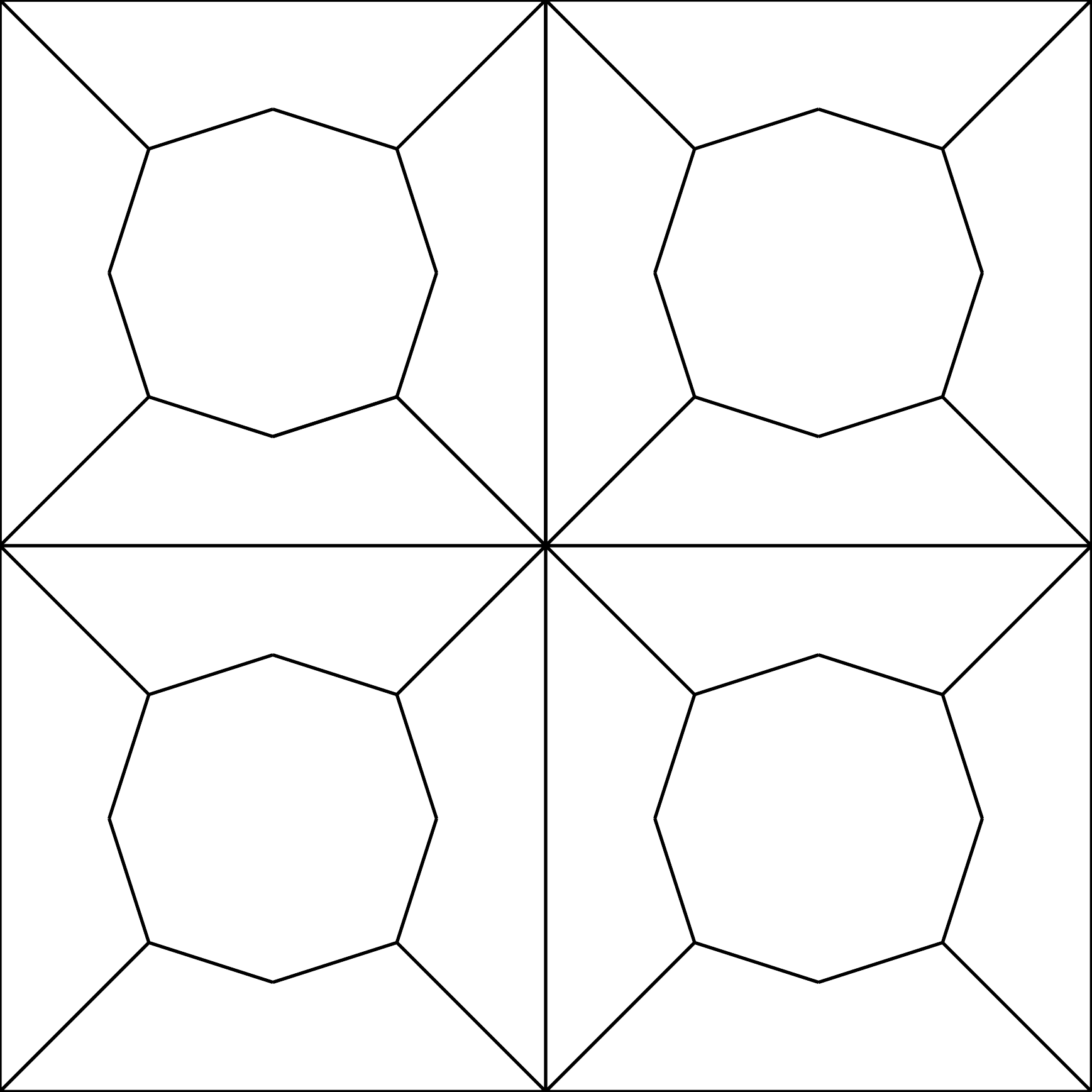}}
 \put(2.2,0){\includegraphics[width=1.5in]{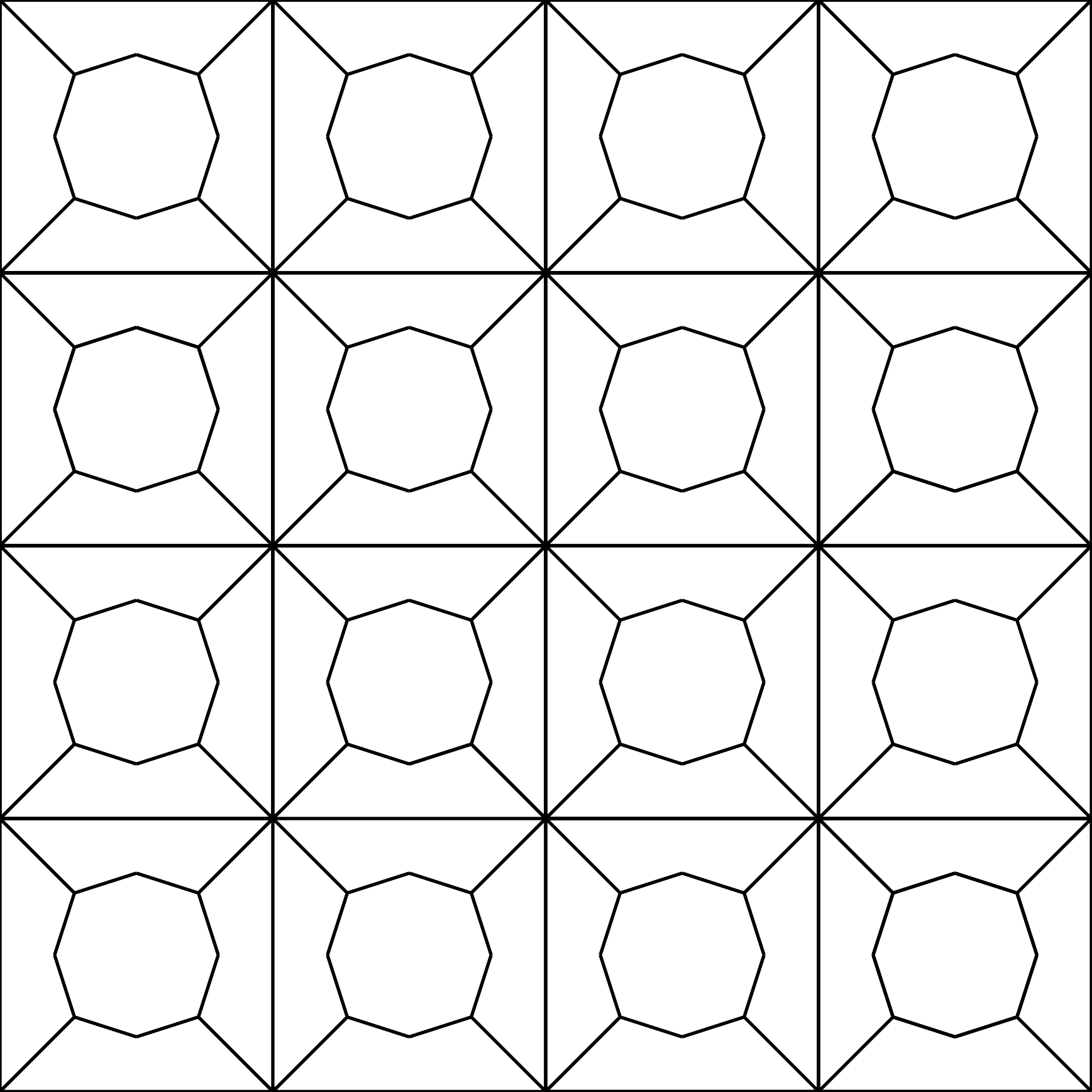}}
    \end{picture}
\caption{ The first three polygonal grids for the computation of Table \ref{t2}.  } \label{5g}
\end{center}
\end{figure}

\begin{table}[h!]
  \centering \renewcommand{\arraystretch}{1.1}
  \caption{Error profiles and convergence rates for \eqref{s1}
        on polygonal grids (Figure \ref{5g}) }\label{t2}
\begin{tabular}{c|cc|cc|cc}
\hline
level & $\|u_h-  u\|_0 $  &rate & $ |u_h-u|_{1,h} $ &rate  & $\3bar u_h- u\3bar $ &rate    \\
\hline
 &\multicolumn{6}{c}{by the $P_2$ conforming DG finite element } \\ \hline
 4&   0.3171E-03 &  1.9&   0.4537E-02 &  1.9&   0.3286E+01 &  1.0 \\
 5&   0.8671E-04 &  1.9&   0.1175E-02 &  1.9&   0.1647E+01 &  1.0 \\
 6&   0.2428E-04 &  1.8&   0.3023E-03 &  2.0&   0.8243E+00 &  1.0 \\ \hline
 &\multicolumn{6}{c}{by the $P_3$ conforming DG finite element } \\ \hline
 1&   0.3402E-02 &  0.0&   0.3868E-01 &  0.0&   0.3702E+01 &  0.0 \\
 2&   0.2027E-03 &  4.1&   0.4895E-02 &  3.0&   0.9408E+00 &  2.0 \\
 3&   0.1476E-04 &  3.8&   0.6244E-03 &  3.0&   0.2368E+00 &  2.0 \\
 \hline
\end{tabular}%
\end{table}%

\end{document}